\documentclass[a4paper]{article}
\usepackage{amsmath,amssymb,amsfonts,amsthm,mathtools,esint,mathrsfs,bm}
\usepackage{caption,subcaption}
\usepackage[english]{babel}
\usepackage[hidelinks]{hyperref}
\usepackage{color}

\usepackage[T1]{fontenc}
\usepackage[utf8]{inputenc}
\usepackage{csquotes}

\newcommand{\abs}[1]{\left\vert{#1}\right\vert}
\newcommand{\norm}[1]{\left\Vert{#1}\right\Vert}

\DeclareMathOperator{\tr}{tr}
\DeclareMathOperator{\dist}{dist}

\newtheorem{thm}{Theorem}
\newtheorem{prop}[thm]{Proposition}
\newtheorem{lem}[thm]{Lemma}
\newtheorem{cor}[thm]{Corollary}
\theoremstyle{definition}
\newtheorem{rem}[thm]{Remark}

\usepackage{framed}

\title{Minimizers of the Landau-de Gennes energy around a spherical colloid particle}
\author{Stan Alama \and Lia Bronsard \and Xavier Lamy}
\date{\today}

\begin{document}
\maketitle

\begin{abstract}
We consider energy minimizing configurations of a nematic liquid crystal around a spherical colloid particle, in the context of the Landau-de~Gennes model.  The nematic is assumed to occupy the exterior of a ball $B_{r_0}$, and satisfy homeotropic weak anchoring at the surface of the colloid and approach a uniform uniaxial state as $|x|\to\infty$.  We study the minimizers in two different limiting regimes:  for balls which are small $r_0\ll L^{\frac12}$ compared to the characteristic length scale $L^{\frac 12}$, and for large balls, $r_0\gg L^{\frac12}$. The relationship between the radius and the anchoring strength $W$ is also relevant. For small balls we obtain a limiting quadrupolar configuration, with a ``Saturn ring'' defect for relatively strong anchoring, corresponding to an exchange of eigenvalues of the $Q$-tensor.  In the limit of very large balls we obtain an axisymmetric minimizer of the Oseen--Frank energy, and a dipole configuration with exactly one point defect is obtained.

\end{abstract}

\section{Introduction}\label{sec:intro}

Liquid crystals are well-known for their many applications in optical devices. The rod-like molecules in a nematic liquid crystal tend to align in a common direction: the resulting orientational order produces an anisotropic fluid with remarkable optical features. This anisotropy also makes it highly interesting to use nematic liquid crystals in colloidal suspensions. Immersion of colloid particles into a nematic system disturbs the orientational order and creates topological defects, which enforce fascinating self-assembly phenomena \cite{poulinstarklubenskyweitz97,musevicetal06}, with many potential applications \cite{senyuketal13,porentaetal14}. This sensitivity to inclusion of small foreign bodies also has promising biomedical applications \cite{woltmanjaycrawford07}: for instance, new biological sensors could detect very quickly the presence of microbes, based on the induced change in nematic order \cite{shiyanovskiietal05,helfinstinelavrentovichwoolverton06,hussainpinaroque09}. 

In the present paper we investigate the structure of the nematic order around one spherical particle, with homeotropic (\textit{i.e.} normal) anchoring at the particle surface, and uniform alignment far away from it. The homeotropic anchoring creates a topological charge. This charge must be balanced in order to match the uniform alignment at infinity, which is topologically trivial. Therefore  one expects to observe singularities. 

This particular problem is a crucial step in understanding more complex situations, and it has received a lot of attention in the past two decades \cite{terentjev95,kuksenoketal96,lubenskyetal98,stark99,ravnikzumer09}. These works rely on heuristically supported approximations, and numerical computations. They point out  two possible types of configurations, with \enquote{dipolar} or \enquote{quadrupolar} symmetry (related to their far-field behavior \cite[§~4.1]{stark01}). In a dipolar configuration the topological charge created by the particle is balanced by a point defect, while in a quadrupolar configuration it is balanced by a \enquote{Saturn ring} defect around the particle.

The aforementioned works use either Oseen-Frank theory \cite{terentjev95,kuksenoketal96,lubenskyetal98,stark99} or Landau-de Gennes theory \cite{ravnikzumer09} to describe nematic alignment. In Oseen-Frank theory, the order parameter is a director field $n(x)\in\mathbb S^2$, which minimizes an elastic energy. One drawback of that model is that line defects have infinite energy. In particular, the energy of a quadrupolar configuration with Saturn ring defect has to be renormalized. Moreover, Oseen-Frank theory only accounts for uniaxial nematic states: it assumes local axial symmetry of the alignment around the average director. On the other hand Landau-de Gennes theory involves a tensorial order parameter that can also describe biaxial states, in which the local axial symmetry is broken. This is the model that we will be using here.

The order parameter in Landau-de Gennes theory is the so-called $Q$-tensor, which belongs to the space
\begin{equation}\label{S0}
\mathcal S_0 :=\left\lbrace Q\in M_3(\mathbb R)\colon Q_{ij}=Q_{ji},\, \tr (Q)=0\right\rbrace,
\end{equation}
of symmetric traceless $3\times 3$ matrices. The eigenvectors of $Q$ represent the average directions of alignment of the molecules, and the associated eigenvalues measure the degree of alignment along these directions. Uniaxial states are described by $Q$-tensors with two equal eigenvalues, which can be put in the form
\begin{equation*}
Q=s\left(n\otimes n-\frac 13 I\right),\quad s\in\mathbb R,\; n\in\mathbb S^2.
\end{equation*}
Biaxial states correspond to the generic case of a $Q$-tensor with three distinct eigenvalues.

The configuration of a nematic material contained in a domain $\Omega\subset\mathbb R^3$ is described by a map $Q\colon \Omega \to \mathcal S_0$. At equilibrium it should minimize the free energy functional
\begin{equation}\label{F}
\mathcal F(Q) =\int_\Omega \left[ \frac L2 \abs{\nabla Q}^2 +f(Q) \right] dx + \mathcal F_s(Q).
\end{equation}
Here $\mathcal F_s(Q)$ is a surface energy term which depends on the type of anchoring (see  \eqref{Fs} below), and the bulk potential $f(Q)\geq 0$ is given by
\begin{equation}\label{potential}
f(Q)=-\frac a 2\,\tr(Q^2)-\frac b 3\,\tr(Q^3)+\frac c 4\,\tr(Q^2)^2 + C_0,
\end{equation}
for some material-dependent constants $a\geq 0$, $b,c>0$. The constant $C_0$ is chosen to ensure that $\min f =0$. The set of $Q\in\mathcal S_0$ minimizing the potential \eqref{potential} obviously plays a crucial role. It consists exactly of those $Q$-tensors which are uniaxial, with  fixed eigenvalues:
\begin{equation}
\mathcal U_* :=\left\lbrace f=0\right\rbrace =\left\lbrace s_*\left( n\otimes n-\frac 13 I\right)\colon n\in\mathbb S^2\right\rbrace,
\end{equation}
where $s_*=(b+\sqrt{b^2+24ac})/4c>0$.

We are interested here in the nematic configuration around a spherical particle:
\begin{equation*}
\Omega = \Omega_{r_0} :=\mathbb R^3\setminus\overline B_{r_0},
\end{equation*}
where $r_0>0$ is the particle radius. We impose uniform $\mathcal U_*$-valued conditions at infinity
\begin{equation}\label{condinfty}
\lim_{\abs{x}\to\infty} Q(x) =Q_\infty := s_* \left(e_z\otimes e_z -\frac 13 I\right),\qquad e_z=(0,0,1).
\end{equation}
At the particle surface, weak radial anchoring is enforced through the surface term $\mathcal F_s$ in the free energy functional \eqref{F}. This surface contribution is given by
\begin{equation}\label{Fs}
\mathcal F_s(Q)=\frac{W}{2} \int_{\partial B_{r_0}} \abs{Q_s-Q}^2\, dA,
\end{equation}
where $W>0$ is the anchoring strength, and $Q_s$ is the $\mathcal U_*$-valued radial map
\begin{equation}\label{Qs}
Q_s := s_*\left(e_r\otimes e_r-\frac 13 I\right),\qquad e_r=\frac{x}{\abs{x}}.
\end{equation}
Denoting by $\nu$ the exterior normal to $\Omega_{r_0}$, the corresponding boundary conditions are
\begin{equation}\label{weakanchor}
\frac{L}{W}\frac{\partial Q}{\partial\nu} = Q_s-Q\quad\text{for }|x|=r_0.
\end{equation}
We also include in this description the case of strong anchoring, corresponding to $W=+\infty$ and Dirichlet boundary conditions
\begin{equation}\label{stronganchor}
Q=Q_s\quad\text{for }|x|=r_0.
\end{equation}
In every case, the Euler-Lagrange equations
\begin{equation}\label{equilibrium}
L \Delta Q =\nabla f(Q) = -a Q -b\left( Q^2-\frac 13 \abs{Q}^2 I\right) +c \abs{Q}^2 Q,
\end{equation}
are satisfied in $\mathcal D'(\Omega_{r_0};\mathcal S_0)$ by any equilibrium configuration.

The existence of minimizers of the free energy functional \eqref{F} with the uniform far-field condition \eqref{condinfty} can be obtained if we replace the pointwise condition \eqref{condinfty} with the integrability condition
\begin{equation}\label{condinftyhardy}
\int_{\Omega_{r_0}} \frac{\abs{Q_\infty -Q}^2}{\abs{x}^2}\, dx <\infty.
\end{equation} 
The Euler-Lagrange equations \eqref{equilibrium} can then be used to see that the strong condition \eqref{condinfty} is in fact also satisfied. After establishing this existence result in Section~\ref{sec:existence},  we turn to studying the two asymptotics regimes of \enquote{small} or \enquote{large} colloid particle.

According to the numerical computations in \cite{stark99,ravnikzumer09}, small particles favor quadrupolar configurations with a defect ring, while large particles favor dipolar configurations with a point defect. In the present paper we obtain rigorous justifications of these observations. In the small particle regime we also provide exact information on the radius of the defect ring, for which the values computed in \cite{terentjev95,kuksenoketal96,lubenskyetal98,stark99,ravnikzumer09} did not agree.

We determine the size of the particle based on the ratio $r_0^2/L$, and consider the limits as the ratio tends to zero and to infinity. In either limiting case, the ratio $r_0W/L$, which gauges the anchoring at the particle surface, will also enter into the description of the limiting configuration. 

\paragraph*{The small particle limit.}
In Section~\ref{sec:small} we investigate the small particle regime, and prove:
\begin{thm}\label{thm:small}
Consider (for any $r_0$, $W$ or $L$) a map $Q$, finite-energy solution of \eqref{equilibrium}-\eqref{weakanchor}-\eqref{condinftyhardy}. Let $w\in (0,+\infty]$ be the effective \enquote{limiting anchoring strength}. Then, as
\begin{equation*}
\left(\frac{r_0^2}{L}, \frac{r_0W}{L}\right) \longrightarrow (0,w),
\end{equation*}
the rescaled maps $x\mapsto Q(r_0 x)$ converge to 
\begin{equation*}
Q_0=s_*\frac{w}{3+w}\frac{1}{r^3}\left(e_r\otimes e_r-\frac 13 I\right)+s_*(1-\frac{w}{1+w}\frac 1r )\left( e_z\otimes e_z-\frac 13 I\right),
\end{equation*}
 locally uniformly in $\overline \Omega_1$.
\end{thm}
An interesting feature of Theorem~\ref{thm:small} is the explicit form of the limit: it provides a very precise description of the quadrupolar configurations. The Saturn ring defect appears as a discontinuity in the principal eigenvector of $Q(x)$, the $Q$-tensor passing through a uniaxial state as eigenvalue branches cross via an \enquote{eigenvalue exchange} mechanism \cite{ravnikzumer09}.  The ratio of the ring radius to the particle size is thus found to be, for $w>\sqrt 3$, the solution $r>1$ of
\begin{equation*}
r^3-\frac{w}{1+w}r^2-\frac{w}{3+w} = 0.
\end{equation*}
For the strong anchoring $w=\infty$, its value is $r\approx 1.47$. As the anchoring strength decreases, the ring shrinks until it becomes a surface ring for $w=\sqrt 3$. At very weak anchoring $w<\sqrt 3$ there is no defect ring anymore. (See Figure~\ref{pictures} in Section~3 for the ring location at various values of $w$.)  This description is consistent with  \cite{terentjev95,kuksenoketal96,lubenskyetal98,stark99,ravnikzumer09}, with the significant improvement of providing exact values for the relevant quantities.  

Away from the ring defect, the limiting map $Q_0(x)$ is everywhere biaxial in $\Omega_1$:  the small particle limit $r_0^2/L \to 0$ does {\em not} correspond to the unit director Oseen-Frank model.  Indeed, it is well known (see \cite{schoenuhlenbeck84}, and \cite{hkl88} for more general energy functionals) that minimizing $\mathbb{S}^2$-valued maps cannot have line defects.

\paragraph*{The large particle limit.}
The large particle regime  is more delicate to analyse. We restrict ourselves to minimizers of the free energy, and to strong anchoring -- that is, $W=\infty$. For the rescaled maps $x\mapsto Q(r_0 x)$, the regime $r_0^2/L\gg 1$ corresponds to the vanishing elastic constant limit studied in \cite{majumdarzarnescu10,nguyenzarnescu13}. There, the authors prove convergence to a $\mathcal U_*$-valued map whose director is an $\mathbb S^2$-valued minimizing harmonic map. It is well-known that such a map has a discrete set of singularities \cite{schoenuhlenbeck84}, and that these defects carry topological degrees $\pm 1$ \cite{breziscoronlieb86}.  In particular, the results of \cite{schoenuhlenbeck84, hkl86} ensure that a ``Saturn ring'' defect cannot be observed in the large particle limit, regardless of the anchoring condition on the particle surface, and thus very different behavior may be expected in the large particle regime than was observed for small particles.

Since the strong radial anchoring imposes a degree $+1$ near the particle, while the uniform far-field condition imposes a zero degree at infinity, there must be at least one point defect of degree $-1$. However the number of defects of a minimizing map does not necessarily correspond to the minimal number of defects required by the topology \cite{hardtlin86}. 
In our case we expect that there is exactly one defect, as predicted by \cite{lubenskyetal98,stark99,ravnikzumer09}. Since determining the exact number of defects is a very difficult question in general, we restrict ourselves to {\em axially symmetric} configurations: we impose invariance under any rotation of vertical axis, and that $e_\theta$ (horizontal unit vector orthogonal to the radial direction) be everywhere an eigenvector of the $Q$-tensor. This natural symmetry assumption seems to be supported by the numerical pictures in \cite{ravnikzumer09}.

In the limit we will therefore obtain an axially symmetric $\mathbb S^2$-valued harmonic map. Such maps have been studied in \cite{hardtkinderlehrerlin90,hardtlinpoon92}. They are analytic away from a discrete set of defects on the $z$-axis. For very particular symmetric boundary data, it can be deduced from rearrangement inequalities that the number of defects matches the topological degree \cite[Theorem~5.1]{hardtkinderlehrerlin90}. This result does not apply to our case, but -- using different arguments -- we nevertheless manage to show that there is exactly one defect, thus justifying the dipole configuration predicted by \cite{lubenskyetal98,stark99,ravnikzumer09}. More precisely, in Section~\ref{sec:large} we prove:
\begin{thm}\label{thm:large}
Let $Q$ minimize the free energy \eqref{F} among axially symmetric maps satisfying the boundary conditions \eqref{stronganchor}-\eqref{condinftyhardy}. Then, as $r_0^2/L$ goes to $+\infty$, a subsequence of the rescaled maps $x\mapsto Q(r_0 x)$ converges to a map
\begin{equation*}
Q_*(x)=s_* ( n(x)\otimes n(x)-I/3),
\end{equation*}
locally uniformly in $\overline\Omega_1\setminus\lbrace p_0\rbrace$.
Here $n$ minimizes the Dirichlet energy in $\Omega_1$, among axially symmetric $\mathbb S^2$-valued maps  satisfying the boundary conditions
\begin{equation*}
n=e_r\;\text{on }\partial B_1,\qquad\text{and } \int_{\Omega_1} \frac{(n_1)^2 +(n_2)^2}{\abs{x}^2}\, dx <\infty,
\end{equation*}
and $n$ is analytic away from \textbf{exactly one point defect} $p_0$, located on the axis of symmetry.
\end{thm}
The core of Theorem~\ref{thm:large} is verifying that the minimizing harmonic map $n$ admits at most one defect. We achieve this by investigating the topology of the sets $\lbrace n_3>0\rbrace$ and $\lbrace n_3 <0\rbrace$ where $n$ points \enquote{more upward} or \enquote{more downward}. Using basic energy comparison arguments and the analyticity of minimizers away from the $z$-axis, we show that these sets are connected. Merging this with the observation that defects correspond to \enquote{jumps} between upward- and downward-pointing $n$, we conclude that there cannot be more than one defect.

We note that there are very few cases in which the number of defects is actually known to match the topological degree. This is true in a ball with radial Dirichlet boundary conditions, because then the energy of the radial map can be explicitly computed and seen to coincide with a general lower bound \cite{breziscoronlieb86}, and this is true also for geometries close enough to the radial one \cite{hardtlin88}.  We do expect that our minimizers in the axially symmetric class are actually minimizers in the general (nonsymmetric) case, but this remains an open question.

The plan of the paper is as follows. In Section~\ref{sec:existence} we prove the existence of minimizers and some basic properties. In Section~\ref{sec:small} we investigate the small particle regime and the quadrupolar \enquote{Saturn ring} configurations. In Section~\ref{sec:large} we study the large particle regime and the associated axially symmetric harmonic map problem.

\paragraph{Acknowledgements:} Part of this work was carried out while XL was visiting McMaster University with a \enquote{Programme Avenir Lyon Saint-Etienne} doctoral mobility scholarship. He thanks the Mathematics and Statistics Department of McMaster University for their hospitality, and his doctoral advisor P.~Mironescu for his support and advice.

\paragraph{Notations:} We will use cylindrical coordinates $(\rho,\theta,z)$ defined by
\begin{equation*}
x_1=\rho\cos\theta,\quad x_2=\rho\sin\theta,\quad x_3=z,
\end{equation*}
and the associated orthonormal frame $(e_\rho,e_\theta,e_z)$, where
\begin{equation*}
e_\rho=(\cos\theta,\sin\theta,0),\quad e_\theta =(-\sin\theta,\cos\theta,0).
\end{equation*}
We will also use spherical coordinates $(r,\theta,\varphi)$ with $(r,\varphi)$ defined by
\begin{equation*}
\rho=r\sin\varphi,\quad z=r\cos\varphi.
\end{equation*}

\section{Existence and first properties of minimizers}\label{sec:existence}

We start by remarking that, with the rescaling $\widetilde Q(x):=Q(r_0x)$, we have:
\begin{equation*}
\frac{1}{r_0^3}\mathcal F(Q) = \int_{\Omega_1}\left[\frac{L}{2r_0^2}|\nabla\widetilde Q|^2 +f(\widetilde Q)\right]\, dx + \frac{W}{2r_0} \int_{\partial B_1}|Q_s-\widetilde Q|^2 dA.
\end{equation*}
We will always work in this rescaled setting, and \textbf{we assume from now on that $\bm{r_0=1}$}, and consider the domain
\begin{equation*}
\Omega :=\Omega_1 =\mathbb R^3\setminus \overline B,
\end{equation*}
where $B=B_1$ is the ball of radius 1. The size of the particle is then encoded in the elastic constant $L$:  $L\to\infty$ represents the small particle limit, while $L\to 0$ simulates very large particle radii.

As mentioned in the introduction, an appropriate functional setting to establish the existence of minimizers is the affine Hilbert space
\begin{equation}\label{Hinfty}
\begin{aligned}
&\mathcal H_\infty :=Q_\infty +\mathcal H,\\
&\mathcal H :=\left\lbrace Q\in H^1_{loc}(\Omega;\mathcal S_0)\colon \int_\Omega \abs{\nabla Q}^2 + \int_\Omega \frac{\abs{Q}^2}{\abs{x}^2} <\infty \right\rbrace.
\end{aligned}
\end{equation}
Note that the free energy functional \eqref{F} is not everywhere finite on the space $\mathcal H_\infty$, since the potential term $f(Q)\geq 0$ may very well not be integrable in $\Omega$. However, since $f(Q_\infty)=0$, we do know that 
\begin{equation*}
\inf_{\mathcal H_\infty} \mathcal F <\infty.
\end{equation*}
In fact, we show that the infimum is attained, and that the minimizer has a limit at infinity:
\begin{prop}\label{prop:existence}
Let $L>0$ and $W\in [0,+\infty]$. Then there exists $Q\in \mathcal H_\infty$ such that
\begin{equation*}
\mathcal F(Q) = \inf_{\mathcal H_\infty}\mathcal F.
\end{equation*}
Moreover, the far-field condition holds in the strong sense \eqref{condinfty}, and this is true for any solution of the Euler-Lagrange equation \eqref{equilibrium} in $\mathcal H_\infty$.
\end{prop}
\begin{rem}
The case $W=+\infty$ will be understood as the strong anchoring case: it amounts to considering only maps $Q\in \mathcal H_\infty$ which satisfy the Dirichlet boundary condition \eqref{stronganchor} in the sense of traces.
\end{rem}
\begin{proof}[Proof of Proposition~\ref{prop:existence}:]
Existence follows from the direct method of the calculus of variations. Thanks to Hardy's inequality, any minimizing sequence $(Q_n)$ is bounded in $\mathcal H_\infty$ and admits (up to taking a subsequence) a weak limit $Q\in \mathcal H_\infty$. We may also assume that the convergence $Q_n\to Q$ holds almost everywhere. Convexity and Fatou's lemma allow us to conclude that $\mathcal F(Q)\leq \liminf \mathcal F(Q_n)=\inf \mathcal F$.

The limit at infinity follows from estimates for solutions of the Euler-Lagrange equation \eqref{equilibrium}. From Lemma~\ref{lem:unifbound} below  we know that $Q\in L^\infty(\Omega)$, and \eqref{equilibrium} readily implies that $\Delta Q\in L^\infty(\Omega)$.
Using standard elliptic estimates, we deduce that $\nabla Q\in L^\infty(\Omega)$, so that $Q$ is uniformly continuous. Since on the other hand Sobolev inequality implies that $\abs{Q-Q_\infty}$ belongs to $L^6(\Omega)$, we conclude that $\abs{Q(x)-Q_\infty}$ converges to zero as $\abs{x}$ goes to $+\infty$.
\end{proof}

In the proof of Proposition~\ref{prop:existence} we used the following $L^\infty$ bound for solutions of \eqref{equilibrium}, related to the growth of the potential $f(Q)$.

\begin{lem}\label{lem:unifbound}
If $Q\in \mathcal H_\infty$ solves \eqref{equilibrium}, it holds
\begin{equation*}
\norm{Q}_{L^\infty(\Omega)}\leq q_0,
\end{equation*}
for some $q_0>0$ that depends only on $a$, $b$ and $c$ (but not on $L$ and $W$).
\end{lem}

\begin{proof}
Let $\widetilde Q :=Q-Q_\infty$, so that $\widetilde Q\in \mathcal H$ solves
\begin{equation}\label{eqtilde}
L\Delta \widetilde Q = \nabla f(Q_\infty + \widetilde Q).
\end{equation}
We may use as a test function in \eqref{eqtilde} any function $\Psi\in \mathcal H$ with compact support in $\overline\Omega$. Let us consider a test function
\begin{equation*}
\Psi=V\widetilde Q,\quad\text{for some }V\geq 0\text{ with compact support in }\overline\Omega.
\end{equation*}
Multiplying \eqref{eqtilde} by $\Psi$ we find
\begin{equation}\label{eqtilde2}
L\int_\Omega \nabla \widetilde Q\cdot\nabla\Psi = -\int_\Omega V \nabla f(Q_\infty + \widetilde Q)\cdot \widetilde Q + \int_{\partial\Omega}V\widetilde Q\cdot\frac{\partial \widetilde Q}{\partial\nu}.
\end{equation}
Clearly there exists $\tilde q_0=\tilde q_0(a,b,c)>0$ such that 
\begin{gather*}
\nabla f(Q_\infty + \widetilde Q)\cdot \widetilde Q \geq 0 \quad\text{for }|\widetilde Q|\geq \tilde q_0,\qquad
\text{and}\quad\abs{Q_s-Q_\infty}\leq \tilde q_0.
\end{gather*}

Now let us take $V$ of the form
\begin{equation}\label{auxV}
V=U\varphi,\quad U=\min\left( (|\widetilde Q|^2-\tilde q_0^2)_+,M \right),\quad 0\leq \varphi\in C_c^\infty(\overline\Omega).
\end{equation}
Note that $V$ is non-negative and  supported inside the set $\lbrace |\widetilde Q|\geq \tilde q_0\rbrace$.

Thanks to the choice of $\tilde q_0$, both terms in the right-hand side of \eqref{eqtilde2} are non-positive: for the first term this is clear, and the second term is zero in the case of strong anchoring \eqref{stronganchor} and non-positive in the case of weak anchoring \eqref{weakanchor} because $|\widetilde Q|^2-\widetilde Q\cdot (Q_s-Q_\infty)\geq 0$ for $|\widetilde Q|\geq \tilde q_0$. Thus we obtain
\begin{equation*}
L\int_\Omega \nabla \widetilde Q\cdot\nabla\Psi \leq 0,
\end{equation*}
i.e.
\begin{equation*}
\int_\Omega \varphi \left( U|\nabla \widetilde Q|^2 +\frac 12 \abs{\nabla U}^2\right) \leq -\int_\Omega U \widetilde Q \cdot \nabla \widetilde Q \cdot \nabla \varphi.
\end{equation*}
Next we take $\varphi=\varphi_R$ such that
\begin{equation*}
\varphi_R(x)=\begin{cases}
1 & \text{ for }\abs{x}\leq R,\\
0 & \text{ for }\abs{x}\geq 2R,
\end{cases}
\qquad\text{and }\abs{\nabla \varphi_R(x) }\leq \frac{C}{\abs{x}},
\end{equation*}
for some constant $C>0$ independent of $R$. We obtain
\begin{equation*}
\int_\Omega \varphi_R \left( U|\nabla \widetilde Q|^2 +\frac 12 \abs{\nabla U}^2\right) \leq MC \|\nabla \widetilde Q\|_{L^2(\abs{x}\geq R)}\|\widetilde Q/r\|_{L^2(\abs{x}\geq R)}.
\end{equation*}
Since $\widetilde Q\in \mathcal H$, the right hand side converges to zero as $R$ goes to $+\infty$ and it holds
\begin{equation*}
\int_\Omega \left( U|\nabla \widetilde Q|^2 +\frac 12 \abs{\nabla U}^2\right) =0
\end{equation*}
Recalling the definition \eqref{auxV} of $U$, we conclude that $|\widetilde Q|\leq \tilde q_0$ a.e., and therefore $\norm{Q}_{L^\infty}\leq \tilde q_0 + s_*\sqrt{2/3}$.
\end{proof}

\begin{rem} It would be interesting to prove an explicit convergence rate for the far-field behavior \eqref{condinfty}. The small particle limit (cf. Section~\ref{sec:small}) suggests a bound of the form
$\abs{Q(x)-Q_\infty}\leq C/\abs{x}$,
for some $C=C(a,b,c,L)>0$.
An indication that this bound could indeed be true is given by the following proposition, which confirms that the minimizers $Q$ approach the uniaxial set $\mathcal U_*$ at the desired rate $|x|^{-1}$.  An analogous situation prevails in the case of solutions of the Ginzburg--Landau equations in the plane \cite{shafrir94}, which have an explicit rate of decay of the complex modulus $|u|-1$, but provides much weaker  asymptotic information in the behavior of the complex phase, the circle $\mathbb{S}^1$ playing the role of $\mathcal{U}_*$ in that setting.
\end{rem}

\begin{prop}\label{prop:decay} 
There exists a constant $C=C(a,b,c)>0$ so that, for any solution $Q\in\mathcal H^\infty$  of \eqref{equilibrium} with finite energy $\mathcal{F}(Q)<\infty$ we have:
\begin{equation}\label{decay}
\dist(Q(x),\mathcal U_*) \leq  \frac{C\sqrt{L}}{\abs{x}} \quad \text{and} 
\quad |\nabla Q(x)|\le \frac{2C}{\abs{x}}.
\end{equation}
\end{prop}

\begin{proof}
We follow the strategy of \cite{shafrir94}, which proved decay estimates for solutions to the Ginzburg--Landau equations in the plane.  Let $A:=\overline{B_4}\setminus B_1$ and for any $R>1$, $A_R:=\overline{B_{4R}}\setminus B_R$.  Define $Q_R(y):= Q(Ry)$ for $y\in A$, 
with energy
$$   \mathcal{E}_R(Q_R; A) = \int_A e_R(Q)\, dy, \quad\text{with}\quad
  e_R(Q)= \frac12 |\nabla Q_R|^2 + {R^2\over L} f(Q_R).  $$
By a change of variables in the integral,
$$
  \mathcal{E}_R(Q_R;A) = {1\over LR}\int_{A_R} \left[ \frac{L}2 |\nabla Q|^2 +  f(Q)\right] dx = o(R^{-1}),
$$
as $R\to\infty$, by the finite energy assumption on $Q$.  Thus, $\nabla Q_R\to 0$ and $R^2\int_A f(Q_R)\, dx \to 0$.  Recalling that $Q(x)\to Q_\infty$ as $|x|\to\infty$ from Proposition~\ref{prop:existence}, we may assume $Q_R\to Q_\infty$ in $H^1(A)$.

We now employ the convergence results for Landau-de~Gennes as $L\to 0$, proven in 
\cite{majumdarzarnescu10}, with our ${L\over R^2}\to 0$ replacing their $L$ in this context.  Although the convergence results in \cite[§~4]{majumdarzarnescu10} are stated for global minimizers of the energy in bounded domains with Dirichlet condition, the proofs of the various convergence lemmas are based on the Monotonicity formula, and apply as well to solutions of the Euler--Lagrange equations \eqref{equilibrium} with uniformly bounded energy, converging in $H^1$-norm.  In particular, we apply Lemma~7 of \cite{majumdarzarnescu10} to $Q_R$ in $\tilde A:= B_3\setminus B_1$, which contains no singularities for sufficiently large $R$:  since for every $a\in A$ we have $\int_{B_1(a)} e_R(Q_R(y)) dy \le o(1),$ by the Lemma
there exists $C_2=C_2(a,b,c)$ for which 
$$\sup_{B_{\frac12}(a)}\left[ \frac12 |\nabla Q_R|^2 + {R^2\over L} f(Q_R)\right] \le C_2.$$
Covering $\tilde A$ by balls of radius 1, we obtain the same uniform bound 
\begin{equation}\label{decayest}
   \frac12 |\nabla Q_R(y)|^2 + {R^2\over L} f(Q_R(y)) \le C_2.  
\end{equation}
for all $y\in \tilde A$.  

Next, we claim that there exists a constant $C_3=C_3(a,b,c)$ such that for any $Q\in \mathcal S_0$, 
\begin{equation}\label{fclaim}
   \left[\dist(Q,\mathcal U_*)\right]^2\le C_3 f(Q).  
 \end{equation}
Indeed, since $f(Q)\gtrsim |Q|^4$ as $|Q|\to\infty$ and $f(Q)=0$ if and only if $Q\in\mathcal{U}_*$, for any fixed $\delta>0$ there exists $C=C(\delta,a,b,c)$ such that the desired bound holds for all $Q\in \mathcal S_0$ with $[\dist(Q,\mathcal U_*)]^2\ge \delta$.  

Since $\mathcal U_*$ is smooth and compact, there exists $\delta>0$ such that for all $Q\in \mathcal S_0$ with $\dist(Q,\mathcal{U}_*)<\delta$ there is a unique orthogonal projection $Q_*\in \mathcal{U}_*$ which minimizes the distance from $Q$ to $\mathcal{U}_*$.  Thus, $Q=Q_* + Q_n$, with $Q_n \perp T_{Q_*}\mathcal{U}_*$, and $|Q_n| = \dist(Q,\mathcal{U}_*)$.  The function $f$ is frame-invariant (it holds $f(R^{-1} Q R)=f(Q)$ for any rotation $R$), so without loss of generality we may assume that $Q_*=Q_\infty=s_*\left(e_z\otimes e_z -\frac13 I\right)$.  Let us consider the orthonormal basis $\lbrace A_j\rbrace$ of $\mathcal S_0$ given by
\begin{gather*}
A_1=\sqrt{\frac 32}(e_z\otimes e_z-I/3),\quad A_2=\frac{1}{\sqrt 2}(e_x\otimes e_x-e_y\otimes e_y),\\
A_3=\frac{1}{\sqrt 2}(e_x\otimes e_y + e_y\otimes e_x),\quad A_4=\frac{1}{\sqrt 2}(e_x\otimes e_z + e_z\otimes e_x),\\
 A_5=\frac{1}{\sqrt 2}(e_y\otimes e_z + e_z\otimes e_y),
\end{gather*}
and identify $Q\in\mathcal S_0$ with $u\in\mathbb R^5$ via
\begin{equation*}
Q=\sum_j u_j\, A_j.
\end{equation*}
In this basis, $T_{Q_*}\mathcal{U}_*=\text{span}\,\{A_4,A_5\},$ and thus we may represent \begin{equation}\label{A123}
Q_n=\sum_{j=1}^3 u_j A_j.
\end{equation}  
Expanding the potential in terms of $u$,
$$
f(Q) = f(Q_*+Q_n)-f(Q_*)= \left(2a+\frac{bs_*}{6}\right)u_1^2 +\frac{bs_*}{2}(u_2^2 +u_3^2)
   + h(u),
$$
where $h(u)=O(|u|^3)$ is a polynomial with terms of third and fourth degree in $u$.  By the representation \eqref{A123} of $Q_n$, we thus have
$$
f(Q)=f(Q_*+Q_n) \ge C_4\left(u_1^2 + u_2^2 + u_3^2\right) = C_4|Q_n|^2=C_4\left[\dist(Q,\mathcal{U}_*)\right]^2,
$$
with constant $C_4=C_4(\delta,a,b,c)$.  The claim is thus established for all $Q\in\mathcal{S}_0$.

Finally, as in \cite{shafrir94} we may return to the original scale $x=Ry$,  to obtain the desired bounds \eqref{decay}.  Indeed, by \eqref{decayest} and $|x|=R|y|\le 3R$,
$$  \frac{1}2 |\nabla Q(x)|^2 + {1\over L}f(Q(x)) 
= R^{-2}\left(\frac1{2} |\nabla Q_R(y)|^2 +{R^2\over L}f(Q_R(y)) \right) \le {C_2\over R^2}
\le {9C_2\over |x|^2},
$$
so the conclusion follows from the above and \eqref{fclaim}.
\end{proof}

\begin{rem}  \rm
From the decay estimate \eqref{decay} it is straightforward (see Proposition~7 of \cite{majumdarzarnescu10}) to show that the eigenvalues of $Q(x)$ tend to those of the uniaxial tensors in $\mathcal{U}_*$, and at the same asymptotic rate $\sqrt{L}/|x|$ as $|x|\to\infty$.
\end{rem}

\section{Small particle: the Saturn ring}\label{sec:small}

This section is dedicated to proving Theorem~\ref{thm:small}, and then studying the limiting configuration $Q_0$, whose expression we recall here:
\begin{equation}\label{Q0}
Q_0=s_*\frac{w}{3+w}\frac{1}{r^3}\left(e_r\otimes e_r-\frac 13 I\right)+s_*(1-\frac{w}{1+w}\frac 1r )\left( e_z\otimes e_z-\frac 13 I\right).
\end{equation}

\begin{proof}[Proof of Theorem~\ref{thm:small}.]
Recall that we assume $r_0=1$, so that we are in fact taking the limit $(1/L,W/L)\to (0,w)$.
It is straightforward to check that the map $Q_0$ \eqref{Q0} belongs to $\mathcal H_\infty$ and solves
\begin{equation}\label{eqQ0}
\left\lbrace
\begin{aligned}
&\Delta Q_0 = 0 \quad\text{in }\Omega,\\
&\frac{1}{w}\frac{\partial Q_0}{\partial\nu} = Q_s-Q_0\quad\text{on }\partial\Omega.
\end{aligned}
\right.
\end{equation}
In what follows we emphasize the dependence of $Q_0$ on the parameter $w>0$ by writing $Q_0=Q_{0,w}$. 
Consider the map $\overline Q := Q-Q_{0,W/L}$, which solves
\begin{equation}\label{eqQbar}
\left\lbrace
\begin{aligned}
&\Delta \overline Q = \frac 1L \nabla f(Q)\quad\text{in }\Omega,\\
& \frac LW \frac{\partial \overline Q}{\partial\nu} + \overline Q = 0 \quad\text{on }\partial\Omega.
\end{aligned}
\right.
\end{equation}
Next we apply an analog of the interpolated estimate of \cite[Lemma~A.2]{bbh93} for the oblique derivative problem (boundary conditions $\partial u/\partial \nu + \alpha u =0$) in $\Omega$: it holds
\begin{equation}\label{interpol}
\norm{\nabla \overline Q}^2_{L^\infty} \leq C \norm{\Delta \overline Q}_{L^\infty} \norm{\overline Q}_{L^\infty}.
\end{equation}
This estimate can be proven exactly as in \cite{bbh93}: apply \cite[Lemma~A.1]{bbh93} which does not need to be adapted, and then follow the proof of \cite[Lemma~A.2]{bbh93} using elliptic $L^p$ estimates for the oblique derivative problem near $\partial\Omega$ (see e.g. \cite[§~15]{agmondouglisnirenberg59} or \cite{chicco72}) instead of the homogeneous Dirichlet problem. The constant $C>0$ depends only on $\delta>0$ such that $w\geq \delta$.

Since we can infer from Lemma~\ref{lem:unifbound} that
$\norm{\nabla f(Q)}_{L^\infty}\leq C$
(where $C>0$ depends only on $a$, $b$ and $c$), the estimate \eqref{interpol} and the equation \eqref{eqQbar} imply that
\begin{equation}\label{gradientQbar}
\norm{\nabla \overline Q}_{L^\infty} \leq \frac{C}{\sqrt L},
\end{equation}
from which we deduce that for any compact $K\subset\overline \Omega$, it holds
\begin{equation*}
\norm{\overline Q}_{L^\infty(K)} \leq C\, \frac{\mathrm{diam}(K)}{\sqrt L},
\end{equation*}
with the constant $C>0$ depending on $a$, $b$, $c$ and $\delta>0$ such that $w\geq \delta$.
On the other hand it can be easily checked that, at any order $k\in\mathbb N$,
\begin{equation*}
\norm{Q_{0,W/L}-Q_{0,w}}_{C^k}\leq c_k \abs{1/w - L/W},
\end{equation*}
so that $Q-Q_0 = \overline Q + Q_{0,W/L}-Q_{0,w}$ converges to zero locally uniformly in $\overline \Omega$, as $(1/L,W/L)\to (0,w)$.
\end{proof}

%
%

The main interest of Theorem~\ref{thm:small} lies in the explicit expression for the limit $Q_0$. Next we investigate its most important features. We start by interpreting the Saturn ring defect as a locus of uniaxiality. More precisely, let $U_w$ be the uniaxial locus of $Q_0=Q_{0,w}$ \eqref{Q0} away from the $z$-axis (on which $e_r=\pm e_z$ and $Q_0$ is trivially uniaxial):
\begin{equation}\label{defUw}
U_w := \left\lbrace x\in\Omega \setminus \mathbb R e_z \colon Q_0(x)\text{ is uniaxial}\right\rbrace.
\end{equation}
Then we have:

\begin{prop}\label{prop:Uw}
For $w>\sqrt 3$ it holds
\begin{equation}\label{Uw}
U_w = \left\lbrace (x_1,x_2,0) \colon x_1^2+x_2^2=r_w^2 \right\rbrace,
\end{equation}
where $r_w$ is the unique solution $r>1$ of
\begin{equation}\label{rw}
r^3-\frac{w}{1+w}r^2-\frac{w}{3+w} = 0.
\end{equation}
The function $w\mapsto r_w$ increases from $r_{\sqrt 3} =1$ to a finite value $r_\infty \approx 1.47$, as $w$ increases from $\sqrt 3$ to $\infty$.

For $w\leq \sqrt 3$, $U_w$ is empty.
\end{prop}

\begin{proof}
Using spherical coordinates $(r,\theta,\varphi)$, the map $Q_0$ is of the form
\begin{equation*}
Q_0 = \alpha(r) (e_r\otimes e_r -I/3) + \beta(r) (e_z\otimes e_z -I/3),\qquad (r>1),
\end{equation*}
where the functions $\alpha(r),\beta(r)>0$ are given by
\begin{equation}\label{alphabeta}
\alpha(r) = s_*\frac{w}{3+w}\frac{1}{r^3} ,\quad\beta(r)=s_*(1-\frac{w}{1+w}\frac 1r ).
\end{equation}
Using the fact that $e_r=\cos\varphi\ e_z +\sin\varphi\ e_\rho$, and defining
\begin{equation*}
\sigma:=\alpha+\beta,\quad\nu:=\alpha\beta,
\end{equation*}
 elementary computations show that the characteristic polynomial of $Q_0$ is
\begin{align*}
P (X) &= \left( X +\frac\sigma 3 \right)\left( X^2-\frac\sigma 3\, X -\frac 29 \sigma^2 +\nu \sin^2\varphi \right)\\
& = \left(X+\frac\sigma 3\right)\left(X-\frac\sigma 6 -\sqrt{\frac{\sigma^2}{4}-\nu\sin^2\varphi} \right)\left(X-\frac\sigma 6 +\sqrt{\frac{\sigma^2}{4}-\nu\sin^2\varphi} \right).
\end{align*}
Note that
\begin{equation*}
\frac{\sigma^2}{4}-\nu\sin^2\varphi =\frac{1}{4}(\alpha-\beta)^2 + \nu\cos^2\varphi \geq 0,
\end{equation*}
so that the above square root is well defined. The eigenvalues $\lambda^0_1 \geq \lambda^0_2 \geq \lambda^0_3$ of $Q_0$ are given by
\begin{align*}
\lambda^0_1 &=\frac\sigma 6 +\sqrt{\frac{\sigma^2}{4}-\nu\sin^2\varphi},\\
\lambda^0_2 & =\frac\sigma 6 -\sqrt{\frac{\sigma^2}{4}-\nu\sin^2\varphi},\\
\lambda^0_3 &= -\frac\sigma 3.
\end{align*}
We are looking for the points where $Q_0$ is uniaxial, \textit{i.e.} either $\lambda^0_1=\lambda^0_2$ or $\lambda^0_2=\lambda^0_3$. For $0<\varphi<\pi$, it holds
\begin{align*}
\lambda^0_1&=\lambda^0_2\quad\Longleftrightarrow\quad \varphi=\frac\pi 2 \text{ and }\alpha=\beta,\\
\lambda^0_2&=\lambda^0_3 \quad\Longleftrightarrow\quad \alpha=0 \text{ or }\beta=0.
\end{align*}
Since $\alpha,\beta>0$, only the first case ($\lambda^0_1=\lambda^0_2$) can occur. Given the expressions \eqref{alphabeta} of $\alpha$ and $\beta$, we deduce that
\begin{gather*}
U_w=\left\lbrace (x_1,x_2,0)\colon 1 < x_1^2+x_2^2 = r^2,\; p(w,r)=0 \right\rbrace,\\
\text{where}\quad p(w,r):=r^3-\frac{w}{1+w}r^2-\frac{w}{3+w}.
\end{gather*}
It is straightforward to check that $r\mapsto p(w,r)$ is increasing on $(1,\infty)$ and therefore $p(w,\cdot)=0$ has at most one solution in this interval. Since on the other hand
\begin{equation*}
(1+w)(3+w)\cdot p(w,1)=3-w^2,
\end{equation*}
there is a solution $r_w>1$  if and only if $w>\sqrt 3$. The function $w\mapsto r_w$ is easily seen to be smooth, and its derivative has the same sign as
\begin{equation*}
-\partial_w p(w,r)=\frac{r^2}{(1+w)^2}+\frac{3}{(3+w)^2}>0,
\end{equation*}
so that $r_w$ is an increasing function of $w$. As $w$ increases to $\infty$, $r_w$ increases to $r_\infty>1$ such that $r_\infty^3 - r_\infty^2 -1=0$.
\end{proof}

\begin{rem}
Since $Q_0$ is biaxial except along the Saturn ring locus $U_w$, there is no conventional nematic director field attached to it.  However, it is reasonable to distinguish a principal direction as an approximate or mean {\it ``director''} field $n_0$, the unit eigenvector 
associated to the largest eigenvalue $\lambda^0_1$. This is well-defined up to a sign, at every point where $\lambda^0_1$ is a simple eigenvalue: that is, everywhere except at $U_w$, where it jumps discontinuously as the eigenvalue branches cross. Then one may compute, for $0<\varphi\leq \pi/2$,
\begin{equation*}
\begin{gathered}
n_0=\sqrt{\frac{1-\mu}{2}}e_\rho +\sqrt{\frac{1+\mu}{2}}e_z,\qquad
\mu :=\frac{\alpha(1-2\sin^2\varphi)+\beta}{\sqrt{\alpha^2+\beta^2 +2\alpha\beta (1-2\sin^2\varphi)}},
\end{gathered}
\end{equation*}
with $\alpha,\beta$ as in \eqref{alphabeta}. For $\pi>\varphi >\pi/2$ the director field is obtained by reflecting with respect to the horizontal plane : $n_0(\varphi)=\sqrt{(1-\mu)/2}\, e_\rho -\sqrt{(1+\mu)/2}\, e_z$. Note that this way $n_0$ is not continuous in $\Omega\setminus U_w$, but it is continuous as an $\mathbb {R P}^2$-valued map : the map $n_0\otimes n_0$ is continuous in $\Omega\setminus U_w$.
\end{rem}

\begin{figure}[ht] 
\begin{center}
\begin{subfigure}{.42\textwidth}
\includegraphics[width=\textwidth]{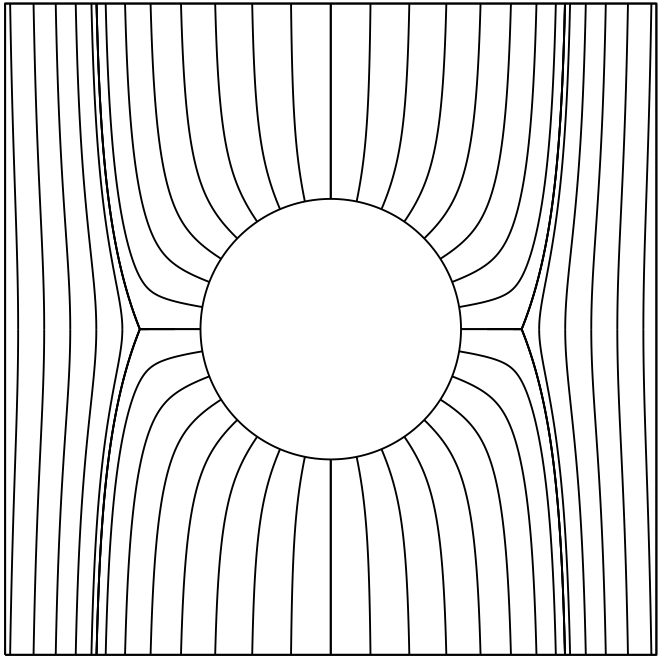}
\caption{$w=\infty$.}
\end{subfigure}
\hspace{.08\textwidth}
\begin{subfigure}{.42\textwidth}
\includegraphics[width=\textwidth]{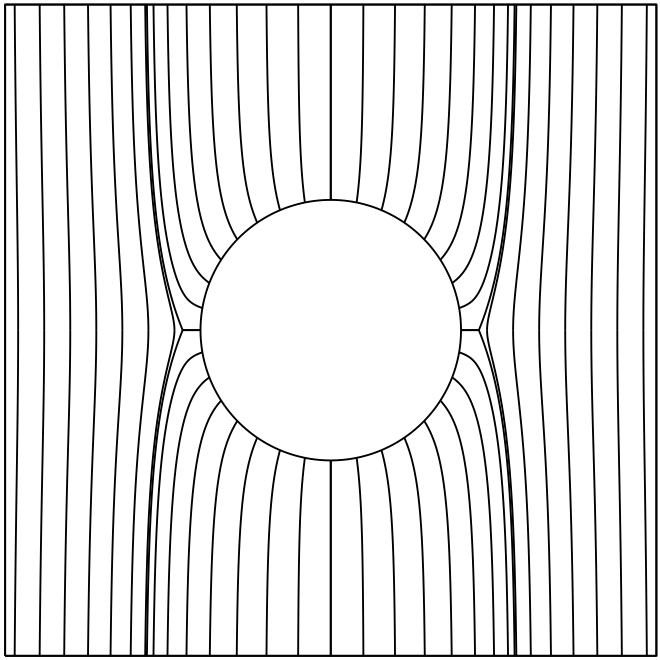}
\caption{$w=3$.}
\end{subfigure}

\vspace{1em}

\begin{subfigure}{.42\textwidth}
\includegraphics[width=\textwidth]{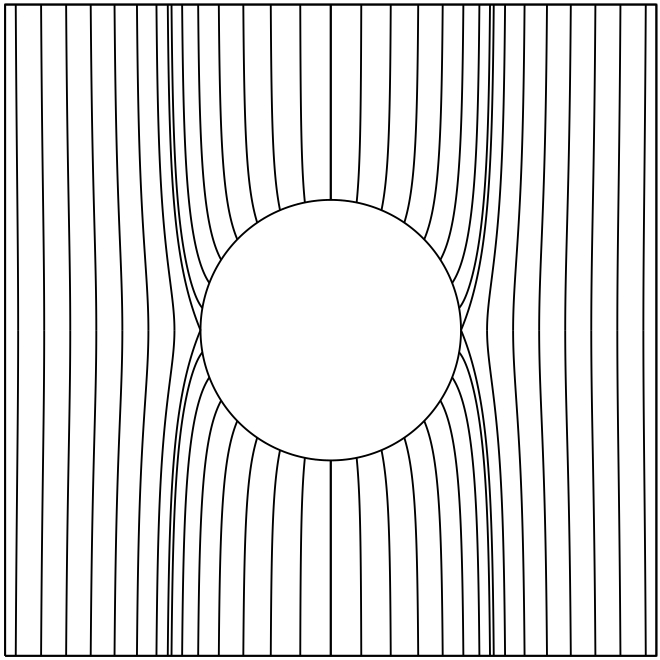}
\caption{$w=1.732\approx \sqrt 3$.}
\end{subfigure}
\hspace{.08\textwidth}
\begin{subfigure}{.42\textwidth}
\includegraphics[width=\textwidth]{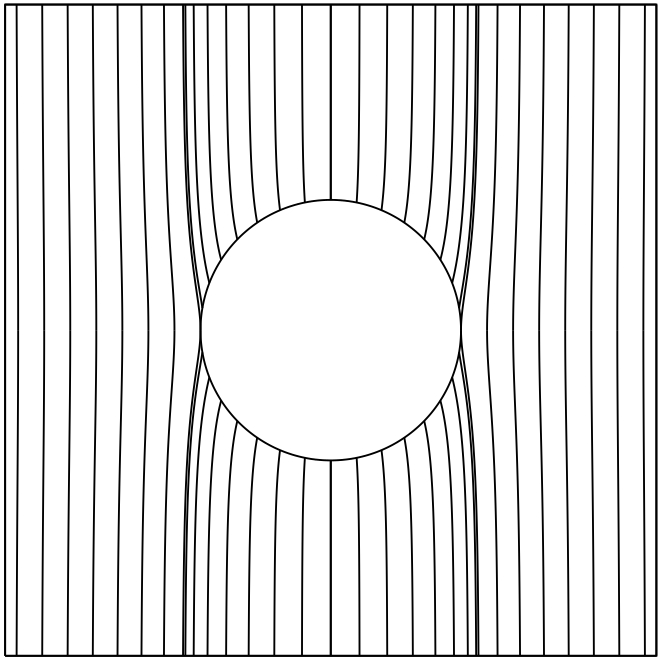}
\caption{$w=1$.}
\end{subfigure}
\caption{The director field $n_0$ (integral curves).}\label{pictures}
\end{center}
\end{figure}

\begin{cor}\label{cor:U}
Let $w>\sqrt 3$.
There exists $\delta_0>0$ such that, if 
\begin{equation*}
\delta :=\frac 1L +\abs{\frac LW -\frac 1w} <\delta_0,
\end{equation*}
then any solution $Q\in\mathcal H_\infty$ of \eqref{equilibrium}-\eqref{weakanchor} admits a uniaxial ring near $U_w$. More precisely, $Q(x)$ is uniaxial for all $x$ belonging to
\begin{equation*}
U=\left\lbrace (\rho_u(\theta)\cos\theta,\rho_u(\theta)\sin\theta,z_u(\theta))\colon \theta\in\mathbb R\right\rbrace,
\end{equation*}
where the functions $\rho_u(\theta)$ and $z_u(\theta)$ satisfy
\begin{equation*}
\abs{\rho_u(\theta)-r_w}+\abs{z_u(\theta)}\leq\varepsilon(\delta)\to 0,
\end{equation*}
as $\delta\to 0$.
\end{cor}

\begin{proof}
Biaxiality can be quantified through the biaxiality parameter \cite{kaiserwiesehess92}
\begin{equation*}
\beta(Q)=1-6\frac{(\tr(Q^3))^2}{\abs{Q}^6},
\end{equation*}
which is such that : $Q$ is uniaxial if and only if $\beta(Q)=0$. Let us fix $R>r_w$. From Theorem~\ref{thm:small} and Proposition~\ref{prop:Uw} we infer that there exists $\delta_0$ such that for $\delta<\delta_0$ it holds
\begin{equation*}
\beta(Q)>0\quad\text{in }A:=\left\lbrace x\in B_R \colon \dist (x,U_w\cup \mathbb R e_z)\geq \varepsilon(\delta)\right\rbrace,
\end{equation*}
for some $\varepsilon(\delta)\to 0$ as $\delta\to 0$. Let us write $\chi_0$, the characteristic polynomial of $Q_0$, as
\begin{equation*}
\chi_0 =(X-\lambda_1^0)P_0,
\end{equation*}
where $P_0=(X-\lambda_2^0)(X-\lambda_3^0)$. By continuity of the roots of a polynomial, it is clear that the characteristic polynomial $\chi$ of $Q$ satisfies
\begin{equation*}
\chi =(X-\lambda_1)P,\quad \abs{\lambda_1-\lambda_1^0}+\abs{P-P_0}\leq c_1(\delta)\quad \text{in }A,
\end{equation*}
for some $c_1(\delta)\to 0$ as $\delta\to 0$. The eigenvalue $\lambda_1$ and the coefficients of $P$ depend continuously on $x\in A$. Note that $P_0(Q_0)n_0\neq 0$ in $A$, and define
\begin{equation*}
u:=\frac{1}{\abs{P_0(Q_0)n_0}}P(Q)n_0.
\end{equation*} 
Then $Q\, u =\lambda_1 u$, and $\abs{u-n_0}\leq c_2(\delta)\to 0$ in $A$, so that $u\neq 0$ and we may define $n=u/\abs{u}$. It holds $Q\, n=\lambda_1 n$ and
\begin{equation*}
\abs{n(x)-n_0(x)}\leq c_3(\delta)\to 0.
\end{equation*}
Moreover the map  $n\otimes n$ is continuous in $A$.

Now fix $\theta\in\mathbb R$ and denote by $H_\theta$ the half plane corresponding to the azimuthal angle $\theta$
\begin{equation*}
H_\theta =\left\lbrace (\rho\cos\theta,\rho\sin\theta,z)\colon \rho\geq 0,\; z\in\mathbb R\right\rbrace,
\end{equation*}
and by $D_\theta$ the disc in $H_\theta$ of radius $\varepsilon(\delta)$ and of center at the point where the ring $U_w$ intersects $H_\theta$:
\begin{equation*}
D_\theta =\left\lbrace (\rho\cos\theta,\rho\sin\theta,z)\colon (\rho-r_w)^2+z^2 \leq\varepsilon(\delta)^2\right\rbrace.
\end{equation*}
We claim that for low enough $\delta$ there exists $x\in D_\theta$ such that $Q(x)$ is uniaxial (which obviously proves Corollary~\ref{cor:U}). 

To prove the claim, note that when restricted to $H_\theta$, the director field $n_0$ may be viewed as an $\mathbb S^1$-valued map since it takes values in $\mathbb S^2\cap H_\theta$. Then the restriction of $n_0\otimes n_0$ to $\partial D_\theta$ is topologically non trivial : it corresponds to a non trival class of $\pi_1(\mathbb{RP}^1)$, as can be seen by explicitly computing its degree. Since for $\delta$ low enough $n(x)$ is arbitrarily close to $n_0(x)$, this implies that the map $n\otimes n$, which is continuous in $A\cap H_\theta$, admits no continuous extension inside $D_\theta$. Therefore $Q$ can not be biaxial everywhere in $D_\theta$ : if it were the case, $Q$ would have only simple eigenvalues and admit a differentiable eigenframe \cite{nomizu73}. In particular there would be a differentiable vector field $\tilde n$ defined in $D_\theta$, such that $\tilde n\otimes \tilde n$ extends $n\otimes n$.
\end{proof}

\section{Large particle: the dipole structure}\label{sec:large}

As mentioned in the introduction, the  core of Theorem~\ref{thm:large} is the fact that the axially symmetric harmonic map obtained in the limit of a large particle has exactly one singularity. In this section, we prove this result (Theorem~\ref{thm:onesing} below) and then complete the proof of Theorem~\ref{thm:large}.

We define the axially symmetric $\mathbb S^2$-valued maps to be exactly the maps $n\in H^1_{loc}(\Omega;\mathbb S^2)$ which can be written in cylindrical coordinates $(\rho,\theta,z)$ as
\begin{equation}\label{npsi}
n = \sin\psi(\rho,z) \, e_\rho + \cos\psi(\rho,z)\, e_z,
\end{equation}
for some real-valued function $\psi\in H^1_{loc}(\Omega_{cyl})$ defined in the domain
\begin{equation}\label{Omcyl}
\Omega_{cyl}:=\left\lbrace (\rho,z)\in\mathbb R_+ \times\mathbb R \colon \rho^2+z^2>1\right\rbrace.
\end{equation}
We consider here strong anchoring conditions given by
\begin{equation}\label{nstronganchor}
n =e_r \quad\text{for }\abs{x}=1,
\end{equation}
and the far-field conditions in integral form
\begin{equation}\label{ncondinftyhardy}
\int_\Omega \frac{(n_1)^2 + (n_2)^2}{\abs{x}^2} dx.
\end{equation}
As in Section~\ref{sec:existence}, the existence of an axially symmetric $\mathbb S^2$ valued map $n$ minimizing the Dirichlet functional 
\begin{equation*}
E(n)=\int_\Omega \abs{\nabla n}^2,
\end{equation*}
under the conditions \eqref{nstronganchor}-\eqref{ncondinftyhardy} follows from the direct method of the calculus of variations and Hardy's inequality. Such a map is analytic away from a discrete set of singularities on the $z$-axis \cite{hardtkinderlehrerlin90}. Here we prove:

\begin{thm}\label{thm:onesing}
Let $n\in H^1_{loc}(\Omega;\mathbb S^2)$ be a minimizer of the Dirichlet functional $E$ among all axially symmetric $\mathbb S^2$-valued maps satisfying \eqref{nstronganchor}-\eqref{ncondinftyhardy}. Then $n$ is analytic away from exactly one point defect: there exists $|z_0|>1$ such that $n$ is analytic in $\overline\Omega\setminus \{(0,0,z_0)\}$.
\end{thm}
\begin{proof}[Proof of Theorem~\ref{thm:onesing}.]

\textbf{Preliminaries.}
Let $\psi\in H^1_{loc}(\Omega_{cyl})$ be the function associated to $n$ through \eqref{npsi}.
 Then $\psi$ minimizes the energy
\begin{equation*}
E(\psi)=\int_{\Omega_{cyl}} \left[ \abs{\partial_\rho \psi}^2 + \abs{\partial_z \psi}^2 +\frac{1}{\rho^2}\sin^2\psi \right]\rho d\rho dz,
\end{equation*}
among all functions $\psi\in H^1_{loc}(\Omega_{cyl})$ such that the corresponding $n$ satisfies \eqref{nstronganchor}-\eqref{ncondinftyhardy}. Next we express these boundary conditions in terms of $\psi$.

The strong anchoring condition \eqref{nstronganchor} is more conveniently expressed using spherical coordinates $(r,\theta,\varphi)$:
\begin{equation}\label{psistronganchor}
\psi = \varphi\quad\text{for }r=1.
\end{equation}
It holds in the sense of traces, which makes sense as soon as $E(\psi)<\infty$. (In fact we should have written $\psi\equiv\varphi$ mod $2\pi$, but since any $\mathbb Z$-valued function of regularity $H^{1/2}$ is constant \cite{bourgainbrezismironescu00} we may reduce to the above.)

The far-field condition \eqref{ncondinftyhardy} becomes
\begin{equation}\label{psicondinftyhardy}
\quad \int_{\Omega_{cyl}}\frac{\sin^2\psi}{\rho^2+z^2} \rho d\rho dz <\infty.
\end{equation}
Hence the class of admissible functions consists exactly of the $\psi\in H^1_{loc}(\Omega_{cyl})$ satisfying $E(\psi)<\infty$ and \eqref{psistronganchor}-\eqref{psicondinftyhardy}.

The Euler-Lagrange equation satisfied by $\psi$ is
\begin{equation}\label{eqpsi}
\partial_z^2 \psi +\partial_\rho^2\psi +\frac{1}{\rho}\partial_\rho\psi =\frac{1}{2\rho^2}\sin(2\psi)\quad\text{in }\Omega_{cyl}.
\end{equation}
Note that, by elliptic regularity, any solution of \eqref{eqpsi} is real-analytic away from the $z$-axis $\lbrace \rho=0\rbrace$.
Also note that, since replacing $\psi$ by $\max(\psi,0)$ or $\min(\psi,\pi)$ does not change the boundary conditions and decreases the energy, it holds
\begin{equation*}
0\leq \psi \leq \pi.
\end{equation*}

The rest of the proof is divided into 3 steps.

\noindent\textbf{Step 1.}
We claim that the open subsets of $\Omega_{cyl}\cap \lbrace \rho>0\rbrace$,
\begin{equation*}
X_+=\lbrace \psi > \pi/2\rbrace\text{ and }X_-=\lbrace \psi <\pi/2\rbrace,
\end{equation*}
 are connected. 

We split the half-circle $\partial\Omega_{cyl}\cap\lbrace \rho>0\rbrace$ into two arcs:
\begin{equation*}
A_\pm = \left\lbrace r=1,\; \varphi=\pi/2 \pm t \colon t\in (0,\pi/2) \right\rbrace.
\end{equation*}
The boundary conditions ensure that $A_\pm\subset X_\pm$. We denote by $\omega_\pm$ the connected component of $X_\pm$ containing  $A_\pm$.

Let us show first that $X_+=\omega_+$.
Consider the function
\begin{equation*}
\widetilde\psi=\begin{cases}
\psi&\quad\text{in }\omega_+,\\
\min(\psi,\pi-\psi)&\quad\text{in }\Omega_{cyl}\setminus\omega_+.
\end{cases}
\end{equation*}
Then it can be checked that $\widetilde\psi\in H^1_{loc}(\Omega_{cyl})$. Moreover, $E(\widetilde\psi)=E(\psi)$ and $\widetilde\psi$ clearly satisfies the boundary conditions \eqref{psistronganchor}-\eqref{psicondinftyhardy}. 
Therefore $\widetilde\psi$ minimizes $E$, and is analytic away from the $z$-axis. In particular, $\min(\psi,\pi-\psi)$ is analytic in $\Omega_{cyl}\setminus \overline\omega_+$.

On the other hand, since $X_-$ is open and non-void, there is an open subset of $\Omega_{cyl}\setminus\overline\omega_+$ in which $\psi<\pi/2$. In that open subset, the two analytic functions $\psi$ and $\min(\psi,\pi-\psi)$ coincide, so they must coincide in the whole $\Omega_{cyl}\setminus\overline\omega_+$. We deduce that 
\begin{equation*}
\psi\leq \pi-\psi\quad\text{i.e.}\quad\psi\leq\pi/2 \qquad\text{in }\Omega_{cyl}\setminus\overline\omega_+,
\end{equation*}
and therefore $X_+=\omega_+$ is connected.

To show that $X_-$ is connected, consider 
\begin{equation*}
\widetilde\psi=\begin{cases}
\psi&\quad\text{in }\omega_-,\\
\max(\psi,\pi-\psi)&\quad\text{in }\Omega_{cyl}\setminus\omega_-.
\end{cases}
\end{equation*} 
As above, $E(\widetilde\psi)=E(\psi)$ and we conclude that
 $X_-=\omega_-$ is connected.
 
\noindent\textbf{Step 2.} There is at most one singularity.

We know \cite{hardtkinderlehrerlin90} that $n$ is analytic in $\Omega$ away from a set of isolated points
\begin{equation*}
Z\subset\lbrace \rho=0,\; \abs{z}>1 \rbrace.
\end{equation*}
In particular $\psi$ is continuous in $\overline\Omega_{cyl}\setminus Z$, and since
\begin{equation*}
\int_{\Omega_{cyl}} \frac{\sin^2\psi}{\rho}d\rho dz \leq E(\psi) <\infty,
\end{equation*}
and $0\leq \psi\leq \pi$, it follows that
\begin{equation*}
\psi \in \lbrace 0,\pi\rbrace \quad\text{on }(\partial\Omega_{cyl} \cap \lbrace \rho=0\rbrace)\setminus Z.
\end{equation*}
At every point $(0,z_0)\in Z$, $\psi$ must be discontinuous, because otherwise $n$ would be continuous around that point (and then real analytic). Therefore it must hold
\begin{equation*}
\text{either }\psi=\begin{cases}
0&\text{ in }(z_0-\delta,z_0),\\
\pi &\text{ in }(z_0,z_0+\delta),
\end{cases}
\quad\text{ or }\psi=\begin{cases}
\pi&\text{ in }(z_0-\delta,z_0),\\
0 &\text{ in }(z_0,z_0+\delta),
\end{cases}
\end{equation*}
for some $\delta>0$.

Let us argue by contradiction and assume that there exist two distinct points
\begin{equation*}
(0,z_1),\, (0,z_2) \in Z,\quad z_1< z_2,\quad [z_1,z_2]\cap Z=\lbrace z_1,z_2\rbrace.
\end{equation*}
There are three cases: either $z_1<z_2<-1$, or $z_1<-1<1<z_2$, or $1<z_1<z_2$. Note that the boundary conditions (together with the boundary regularity of $n$ \cite{hardtkinderlehrerlin90}) ensure that $\psi=\pi$ in $(-1-\delta,-1)$ and $\psi=0$ in $(1,1+\delta)$ for some $\delta>0$. In all three cases, it is easy to see that there must exist four distinct points 
\begin{equation*}
(0,z'_j)\notin Z,\quad z'_1<z'_2<z'_3<z'_4,
\end{equation*}
such that
\begin{gather*}
\text{either}\quad
\psi(z'_1)=\psi(z'_3)=0,\; \psi(z'_2)=\psi(z'_4)=\pi,\\
\quad\text{ or}\quad
\psi(z'_1)=\psi(z'_3)=\pi,\; \psi(z'_2)=\psi(z'_4)=0.
\end{gather*}
We may assume that we are in the first case (the second case can be dealt with similarly). Then by continuity there exists $\delta>0$ such that
\begin{gather*}
\left[B((0,z'_1),\delta)\cap\overline\Omega_{cyl}\right] \cup \left[B((0,z'_3),\delta)\cap\overline\Omega_{cyl}\right]\subset X_-,\\
\left[B((0,z'_2),\delta)\cap\overline\Omega_{cyl}\right] \cup \left[B((0,z'_4),\delta)\cap\overline\Omega_{cyl}\right]\subset X_+.
\end{gather*}
Since the sets $X_\pm$ are path-connected we deduce from the above that there is a continuous path $\gamma_-$ from $z'_1$ to $z'_3$ inside $X_-$, and a continuous path $\gamma_+$ from $z'_2$ to $z'_4$ inside $X_+$. Since $z'_1<z'_2<z'_3<z'_4$, these paths must intersect, but then there intersection would belong to $X_-\cap X_+=\emptyset$. This contradiction shows that $Z$ contains at most one point.

\noindent\textbf{Step 3.} There is at least one singularity.

Assume that $Z$ is empty. Then $n$ is continuous in $\overline\Omega$, and therefore 
\begin{equation*}
\deg (n,\partial B_r) = -\frac 12 \int_0^\pi \partial_\varphi \psi(r,\varphi) \sin\psi(r,\varphi)\, d\varphi,
\end{equation*}
is independent of $r\geq 1$. We deduce that
\begin{align*}
1& =\deg (n,\partial B_1)  =\deg(n,\partial B_r)\\
& \leq \frac 14 \int_0^\pi \frac{\sin^2\psi(r,\varphi)}{\sin^2\varphi}\sin\varphi d\varphi +\frac 14 \int_0^\pi \partial_\varphi \psi(r,\varphi)^2 \,\sin\varphi d\varphi,
\end{align*}
which implies $E(\psi)\geq 4\int_1^\infty dr =\infty$, a contradiction.
\end{proof}

Before turning to the proof of Theorem~\ref{thm:large}, we define rigorously the axially symmetric $Q$-tensor maps.  They are the maps $Q\in\mathcal H_\infty$ which satisfy the two following natural symmetry constraints:
\begin{itemize}
\item the map $Q$ is invariant by rotation around the vertical axis:
\begin{equation}\label{invarianceaxial}
Q(Rx) ={}^t \! R Q(x) R\quad\text{ for all rotations }R\text{ of axis }e_z,
\end{equation}
\item the vector $e_\theta$ is everywhere an eigenvector of $Q$:
\begin{equation}\label{ethetaeigen}
Q(x)e_\theta \cdot e_\rho = Q(x) e_\theta \cdot e_z \equiv 0.
\end{equation}
\end{itemize}
These constraints are natural, in the sense that a minimizer of the free energy \eqref{F} under those restrictions is still a solution of the complete (unconstrained) Euler-Lagrange system \eqref{equilibrium} -- as can be easily checked. We denote this class of symmetric maps by $\mathcal H_\infty^{sym}\subset\mathcal H_\infty$.


\begin{proof}[Proof of Theorem~\ref{thm:large}.]
Recall that we assume $r_0=1$ and we are therefore considering the limit $1/L\to\infty$.
Since axial symmetry \eqref{invarianceaxial}-\eqref{ethetaeigen} is clearly preserved by pointwise convergence, we may proceed exactly as in \cite[Proposition~1]{bbh93} to obtain  a subsequence $(Q_k)$ converging in $\mathcal H_\infty$ to an axially symmetric $\mathcal U_*$-valued map $Q_*$ minimizing the Dirichlet energy $\int_\Omega \abs{\nabla Q}^2$. The estimates in \cite{majumdarzarnescu10,nguyenzarnescu13} show that the convergence is in fact locally uniform in $\overline\Omega$, away from the singular set of $Q_*$.

Because $\Omega$ is simply connected, $\mathcal U_*$-valued $H^1_{loc}$ maps can be lifted to $\mathbb S^2$-valued $H^1_{loc}$ maps \cite{ballzarnescu11}: for any $\mathcal U_*$-valued  $Q\in\mathcal H_\infty$, there exists $n\in H^1_{loc}(\Omega;\mathbb S^2)$ such that
\begin{equation*}
Q=Q_n := s_*(n\otimes n-I/3).
\end{equation*}
Therefore, the limiting map $Q_*$ can be written as $Q_*=Q_n$, and the map $n$ minimizes the Dirichlet energy in the class
\begin{equation*}
\mathcal H_*^{sym} :=\left\lbrace n\in H^1_{loc}(\Omega;\mathbb S^2)\colon Q_n\in \mathcal H_\infty^{sym} \text{ with strong anchoring }\eqref{stronganchor} \right\rbrace.
\end{equation*}

To conclude the proof, it remains to show that the class $\mathcal H_*^{sym}$ actually corresponds to the class considered in Theorem~\ref{thm:onesing}. 

Using the fact that $\abs{n}^2=1$, we calculate
\begin{equation*}
s_*^{-2}\abs{Q_n-Q_\infty} = 2(n_1^2 +n_2^2),
\end{equation*}
so that $Q_n$ satisfying \eqref{condinftyhardy} translates into $n$ satisfying \eqref{ncondinftyhardy}.

The strong anchoring condition \eqref{stronganchor} for $Q_n$ is equivalent to
\begin{equation*}
n_{|\partial B} =\tau \, e_r,
\end{equation*}
for some $\lbrace \pm 1\rbrace$-valued function $\tau$, which must be of regularity $H^{1/2}$ and therefore constant \cite{bourgainbrezismironescu00}. Therefore, up to multiplying the map $n$ by a sign, the strong anchoring condition becomes \eqref{nstronganchor}.

The fact that $Q_n$ admits $e_\theta$ as an eigenvector \eqref{ethetaeigen} is equivalent to
\begin{equation*}
(n\cdot e_\theta)(n\cdot e_\rho) = (n\cdot e_\theta)(n\cdot e_z) =0 \quad\Leftrightarrow\quad n\cdot e_\theta \in \lbrace 0,\pm 1\rbrace.
\end{equation*}
Since the function $n\cdot e_\theta$ is $H^1_{loc}$, it must therefore be constant. The boundary conditions prevent it to be equal to $\pm 1$, so that $n\cdot e_\theta \equiv 0$.

The invariance by rotation \eqref{invarianceaxial} for $Q_n$ is equivalent to $n(Rx)=\pm Rn(x)$ for any rotation $R$ of axis $e_z$. The sign $\pm 1$ may depend on $x$ and on $R$, but the $H^1_{loc}$ regularity implies that it does not depend on $x$. Therefore it holds, using cylindrical coordinates,
\begin{equation*}
n(\rho,\theta,z)=\tau(\theta)\, R_\theta\, n(\rho,0,z),\quad\tau(\theta)=\pm 1,
\end{equation*}
where $R_\theta\in SO(3)$ is the rotation of axis $e_z$ and angle $\theta$. The function $\tau$ is easily seen to belong to $H^{1/2}(\mathbb R/2\pi\mathbb Z)$: we conclude that $\tau\equiv 1$.

Therefore the axial symmetry \eqref{invarianceaxial}-\eqref{ethetaeigen} of $Q_n$ is equivalent to
\begin{equation*}
n(\rho,\theta,z)=R_\theta \, n(\rho, 0,z),\quad \text{and}\quad n\cdot e_\theta \equiv 0,
\end{equation*}
which implies that 
\begin{equation*}
n(\rho,\theta,z)=u_1(\rho,z) \, e_\rho + u_2(\rho, z) e_z,\quad u\in H^1_{loc}(\Omega_{cyl};\mathbb S^1).
\end{equation*}
Since $\Omega_{cyl}$ is simply connected, $u$ can be lifted to a real-valued function $\psi\in H^1_{loc}(\Omega_{cyl})$: $u_1=\sin\psi$, $u_2=\cos\psi$. Therefore $n$ is of the form \eqref{npsi}, and $\mathcal H_*^{sym}$ corresponds indeed (up to a sign) to the class of axially symmetric $\mathbb S^2$-valued maps satisfying \eqref{nstronganchor}-\eqref{ncondinftyhardy}.
\end{proof}

\bibliographystyle{plain}
\bibliography{saturn}

\end{document}